\documentclass[11pt, leqno]{amsart}
\usepackage{amssymb, amsmath, amsthm} 
\usepackage{latexsym} 
\usepackage{amscd}
\usepackage{graphicx}
\usepackage{graphics}
\usepackage{url}

\newcommand{\Z}{\mathbb Z}

\newcommand{\R}{\mathbb{R}}

\theoremstyle{plain}

\newtheorem{theorem}[subsection]{Theorem}
\newtheorem{proposition}[subsection]{Proposition}
\newtheorem{lemma}[subsection]{Lemma}
\newtheorem{corollary}[subsection]{Corollary}

\title[An additive problem]{The additive problem for the number of representations as a sum of two squares}
\author{Fernando Chamizo}
\thanks{The author is partially supported by the MTM2017-83496-P grant of the
MICINN (Spain) and by ``Severo Ochoa Programme for Centres of Excellence in
R{\&}D'' (SEV-2015-0554).}
\address{Departamento de Matem\'{a}ticas and ICMAT. Universidad Aut\'{o}noma de Madrid. 28049 Madrid, Spain}

\begin{document}

\begin{abstract}
 We improve a previous unconditional result  about the asymptotic behavior of $\sum_{n\le x} r(n)r(n+m)$ with $r(n)$ the number of representations of $n$ as a sum of two squares when $m$ may vary with $x$.
\end{abstract}

\maketitle


\section{Introduction}
We consider the analogue of the additive divisor problem when the usual divisor function $\tau(n)$ is replaced by $r(n)$, the number of representations of~$n$ as a sum of two squares, which is related to the divisor function for the Gaussian integers. Namely, we study the asymptotic behavior of
\begin{equation}\label{S_def}
 S(x,m)
 =
 \sum_{n\le x} r(n)r(n+m)
 \qquad\text{with}\quad m\in\Z^+.
\end{equation}
With a broad view this can be considered a shifted  convolution of theta coefficients. 
\smallskip

Apparently Estermann was the first author considering this problem~\cite{estermann}. 
His result implies that for $m$ fixed and $2^k$ the $2$-part in the factorization of~$m$
\begin{equation}\label{S_asymp}
 S(x,m)\sim
 8\big|2^{k+1}-3\big|
 \sigma\Big(\frac{m}{2^k}\Big)
 \frac{x}{m}
 \qquad\text{as}\quad x\to\infty.
\end{equation}
Actually in the original paper \cite{estermann} the coefficient of $x/m$ is expressed in the more compact form $8\sum_{d\mid m}(-1)^{m+d} d$ which arguably gives less insight about its size for large values of $m$ and about the role of the powers of~2.  

Estermann's paper is hard to read (it includes a list of more than 40 abbreviations for limits of summation) but the underlying idea is clear: Write $S(x,m)$ in terms of Kloosterman sums and use individual bounds for them. The obtained error term is weak and it can be readily improved using Weil's bound not available at that time. Following \cite{iwaniec}, a more compact and stronger approach without any reference to Kloosterman sums  is to interpret $S(x,m)$ as a hyperbolic circle problem and use the spectral expansion of automorphic kernels. 
An old unpublished result by Selberg (cf. \cite{selberg}) gives the error term~$O(x^{2/3})$ for the hyperbolic circle problem. It is even today the best known result and it translates into a similar error term for $S(x,1)$ \cite{iwaniec}. 
\medskip

Here we address the size of the error term and its uniformity in $m$ to study to what extent $m$ can depend on $x$ keeping \eqref{S_asymp} valid. This problem was treated in \cite{chamizo_c}. Most of the results were stated there under a certain conjecture on spectral theory but the last section includes some unconditional results. In connection with the asymptotic formula, it was proved \cite[Cor.5.3]{chamizo_c} that \eqref{S_asymp} is valid if $m=m(x)$ satisfies $m=O\big(x^{17/11-\epsilon}\big)$ for some $\epsilon>0$. 

Our main result approximates $S(x,m)$ improving the bound for the error term given in \cite{chamizo_c}.
We prefer to establish it in terms of the optimal exponent for the Hecke eigenvalues $\lambda_j(m)$ of the Maass-Hecke waveforms (see the next section for more on the notation). This is defined as 
\begin{equation}\label{lth}
 \theta=
 \inf
 \big\{
 t\ge 0\,:\,
 |\lambda_j(m)|= O\big( m^t\big)
 \big\}.
\end{equation}
The celebrated Ramanujan-Petersson conjecture \cite{BlBr} claims $\theta=0$  in the stronger form $\lambda_j(m)|\le \tau(m)$ but this is out of reach with current methods. 
\smallskip

Our main result is: 

\begin{theorem}\label{main}
 Define $E(x,m)$ to be the error term in \eqref{S_asymp}
 \[
  E(x,m)
  =
  S(x,m)-
  8\big|2^{k+1}-3\big|
  \sigma\Big(\frac{m}{2^k}\Big)
  \frac{x}{m}.
 \]
 Then we have for every $\epsilon>0$
 \[
  m^{-\epsilon}E(x,m)
    \ll
  \begin{cases}
   x^{2/3} & \text{if } m^{3/2+3\theta}\le x,
   \\
   m^{(1+2\theta)/4}x^{1/2} & \text{if } m^{3/2-\theta}\le x<m^{3/2+3\theta},
   \\
   m^{2\theta/3}x^{2/3} & \text{if } m\le x<m^{3/2-\theta},
   \\
   m^{(1+2\theta)/3}x^{1/3} & \text{if } x<m.
  \end{cases}
 \]
\end{theorem}
If $|\lambda_j(p)|=O\big(p^{\theta_0}\big)$ for $p$ prime, the multiplicative properties of $\lambda_j(m)$ imply $\theta\le\theta_0$ and Theorem~\ref{main} is valid replacing $\theta$ by $\theta_0$. In this case $m^{\epsilon}$ could be substituted by certain power of $\tau(m)$. In particular if $m$ is assumed to have a bounded number of divisors then we can take $\epsilon=0$.

\begin{corollary}
 Under Ramanujan-Petersson conjecture
 \[
  E(x,m)\ll x^{2/3}+m^{1/3+\epsilon} x^{1/3}
  \qquad\text{for every $\epsilon>0$}.
 \]
\end{corollary}

\begin{corollary}\label{munif}
 The asymptotic formula \eqref{S_asymp} holds for $m=m(x)$ satisfying  $m=O\big(x^{\eta}\big)$ for  $\eta<2/(1+2\theta)$. 
In particular, Ramanujan-Petersson conjecture would allow to take  any $\eta<2$.
\end{corollary}
\smallskip

Unfortunately with our present knowledge about $\theta$ we cannot rule out the possibility of $\theta$ belonging  to a very thin interval of length $1/192$ in which one can carry out a different optimization getting a slight improvement in a small range. To keep the statement simpler we did not include it in Theorem~\ref{main}. 

\begin{theorem}\label{maini}
 If $\theta\in (5/48, 7/64]$, for $m^{11(1-4\theta)/7}<x<m^{\min(92\theta-5,\,46\theta)/5}$ the bound in Theorem~\ref{main} can be improved to 
 \[
  E(x,m)
    \ll
  x^{17/23+\epsilon}
  +
  (mx)^{17/46+\epsilon}
  +
  m^{(13+4\theta)/28}x^{1/4+\epsilon}
 \]
 for every $\epsilon>0$.
\end{theorem}

To put into perspective the numerical size of the improvement of Theorem~\ref{main} in the range indicated in Theorem~\ref{maini}, we mention that 
the maximal gain is $x^{4/1137}$ attached only when $\theta=7/64$ and $m=x^{1232/1137}$. For instance for $m=x$ the saving is at most $x^{1/2208}$.
Note that the interval for $x$ collapses as $\theta\to 5/48$.

When we substitute the best known upper bound for $\theta$, due to H.~Kim and Sarnak (see Lemma~\ref{kisal} below), Theorem~\ref{main} and Theorem~\ref{maini} show:

\begin{corollary}\label{submain}
 With the notation as before
 \[
  m^{-\epsilon}E(x,m)\ll 
  \begin{cases}
   x^{2/3} & \text{if } m^{117/64}\le x,
   \\
   m^{39/128}x^{1/2} & \text{if } m^{89/64}\le x<m^{117/64},
   \\
   m^{7/96}x^{2/3} & \text{if } m^{161/160}\le x<m^{89/64},
   \\
   x^{17/23} & \text{if } m\le x<m^{161/160},
   \\
   (mx)^{17/46} & \text{if } m^{1137/1232}\le x<m,
   \\
   m^{215/448}x^{1/4} & \text{if } m^{99/112}\le x<m^{1137/1232},
   \\
   m^{13/32}x^{1/3} & \text{if } x<m^{99/112}.
  \end{cases}
 \]
\end{corollary}
In the first of the following figures it is represented the graph of the piecewise linear function $\beta= \beta(\alpha)$ 
such that 
$x^{-\epsilon}E(x,x^{1/\alpha})\ll x^\beta$ is the corresponding bound in   Corollary~\ref{submain}. 
The vertical dashed lines mark  the change from a linear function to another. 
The second figure shows the detail of the range in which the improvement of Theorem~\ref{maini} applies. The thinner line represents the result from Theorem~\ref{main} not taking into account this improvement. 
\begin{center}
 \begin{tabular}{c}
  \includegraphics[height=123pt]{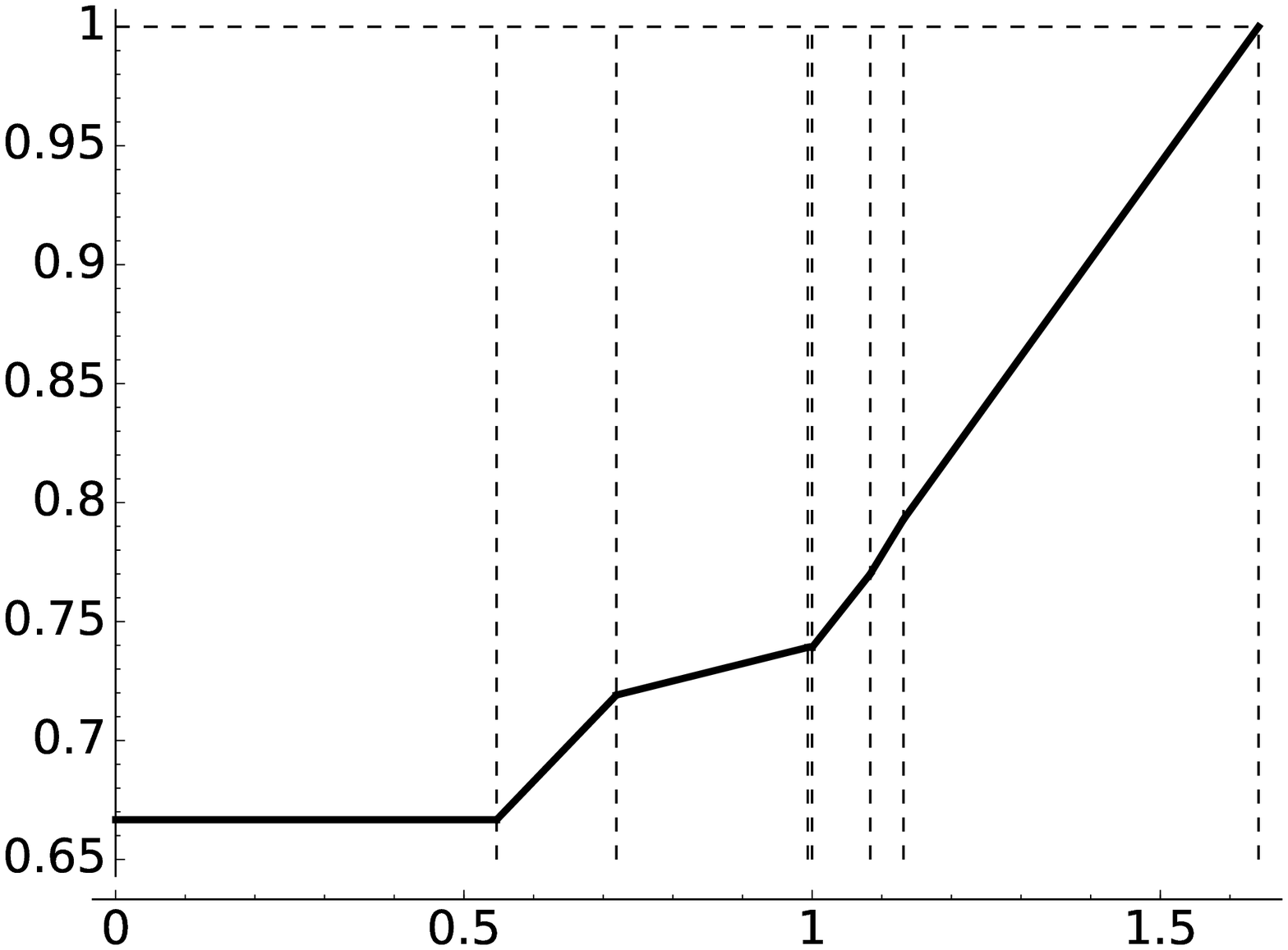}
 \end{tabular}
 \ 
 \begin{tabular}{c}
  \includegraphics[height=123pt]{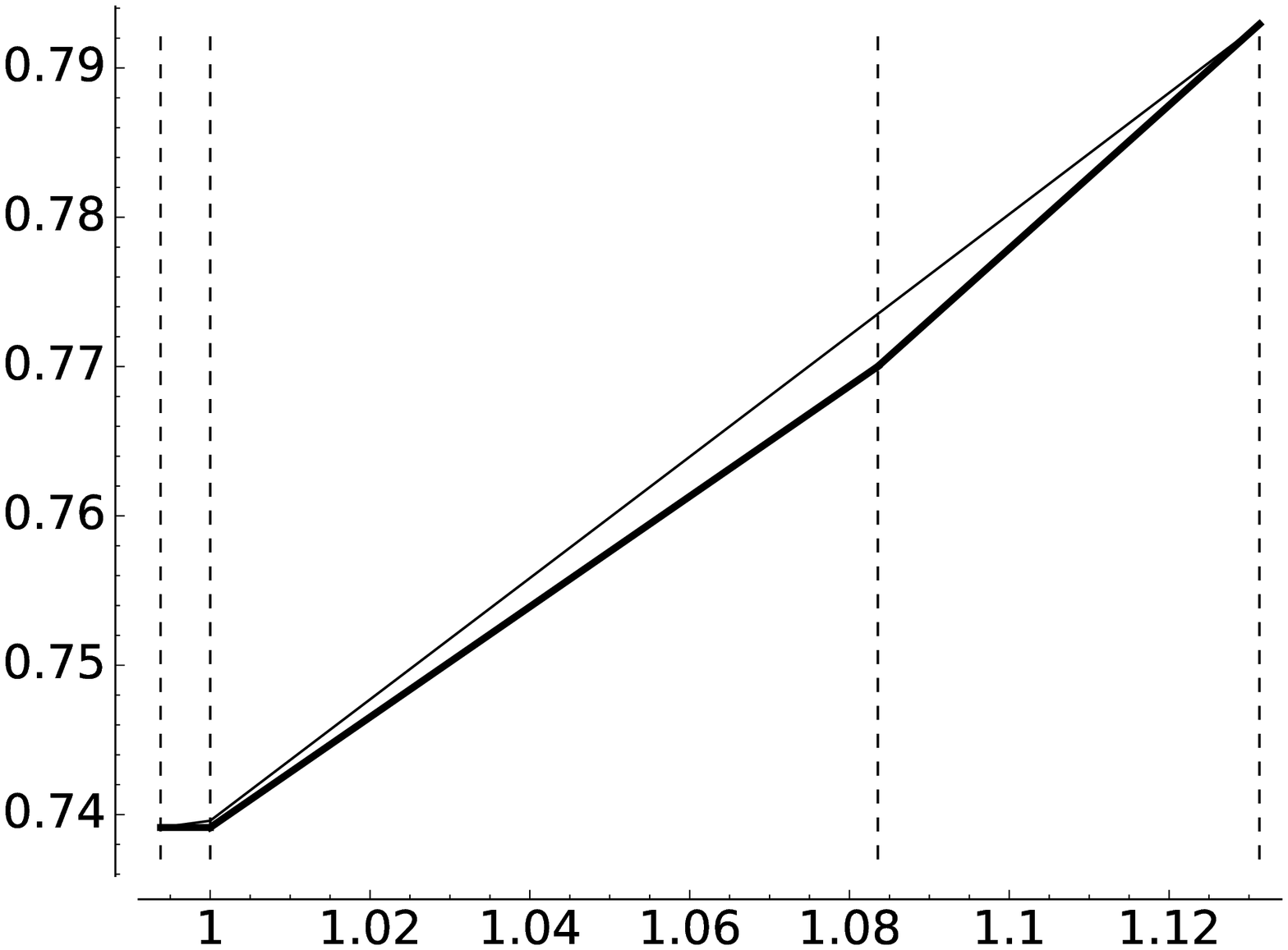}
 \end{tabular}
\end{center}

Corollary~\ref{submain} allows to improve the uniformity achieved in \cite[Cor.5.3]{chamizo_c}.

\begin{corollary}\label{c64_39}
 The asymptotic formula \eqref{S_asymp} holds for $m=m(x)$ satisfying $m=O\big(x^{\eta}\big)$ for  $\eta<64/39$. 
\end{corollary}

\section{Preliminary results}\label{prelim}

The aim of this section is to introduce some notation and to state Proposition~\ref{spec} that embodies all the information about the spectral expansion of $S(x,m)$. 
We also recall some known results about the Hecke eigenvalues and we finish with some comments about the spectral meaning of~\eqref{S_asymp}.
The material of this section is essentially included in \cite{chamizo_c} and it is recalled here for the sake of completeness and to fix the notation. 
\medskip

Consider the set 
\[
 \mathcal{C}(x,m)=
 \big\{
 (a,b,c,d)\in\Z^4\,:\,
 c^2+d^2\le x
 \text{ and }
 a^2+b^2-c^2-d^2=m
 \big\}.
\]
It is plain that $S(x,m)$ in \eqref{S_def} and $\#\mathcal{C}(x,m)$ coincide. Consider also the allied quantity
\[
 A(x,m)
 =
 \#\big\{ (a,b,c,d)\in\mathcal{C}(x,m)
 \,:\, 
 2\mid a-c,\ 2\mid b-d
 \big\}.
\]

For $m$ even the formula
\[
 S(x,m)
 =
 \begin{cases}
  S(x/2,m/2) &\text{if }4\nmid m,
  \\
  2A(x,m)-S(x/2,m/2) &\text{if }4\mid m
 \end{cases}
\]
is elementary (see the proof of Lemma~2.2 in \cite{chamizo_c} for a detailed discussion). Clearly if $8\mid m$ then the last $S(x/2,m/2)$ can be reduced once more. In general we have

\begin{lemma}\label{SA}
 If $m$ is even and $2^k\mid m$ with $k$ maximal then 
 \[
  S(x,m)
  =
  2\sum_{j=0}^{k-1}
  (-1)^j
  \widetilde{A}\big( x/2^j, m/2^j \big)
 \]
 where $\widetilde{A}(x,m)=A(x,m)$ if $4\mid m$ and 
 $\widetilde{A}(x,m)=\frac 12 S(x/2,m/2)$ otherwise.
\end{lemma}

This is a re-formulation of the first part of \cite[Lemma~2.2]{chamizo_c}. 
It implies that to deal with \eqref{S_asymp} it is enough to consider the
asymptotic behavior of $S(x,m)$ for $2\nmid m\in\Z^+$ and that of $A(4x,4m)$ for $m\in\Z^+$.
It turns out that both expressions have a similar spectral expansion.  This is the content of the following proposition that summarizes \S2 of 
\cite{chamizo_c} (see this reference for the proof).

\begin{proposition}\label{spec}
 The quantities
 $A(4x,4m)$ for $m$ a positive integer
 and
 $S(x,m)$ for $2\nmid m\in\Z^+$
 admit spectral expansions of the form 
 \[
  2\sqrt{m}\sum_j \lambda_j(m) h(t_j)|u_j(z_0)|^2
  +
  \frac{\sqrt{m}}{2\pi}
  \sum_{\mathfrak{a}}
  \int_{\R}
  \eta_{t}(m)
  h(t)
  \big|E_{\mathfrak{a}}(z_0,\frac{1}{2}+it)\big|^2\; dt
 \]
 where $\{u_j\}$ is a complete collection of normalized Maass-Hecke waveforms (including the constant eigenfunction) with Hecke eigenvalues $\lambda_j(m)$, $E_{\mathfrak{a}}$ are the Eisenstein series and $\eta_{t}(m)$ their respective Hecke eigenvalues. Here the underlying Fuchsian group $\Gamma$ and the point $z_0$ are 
 $\Gamma=\text{\rm PSL}_2(\Z)$, $z_0=i$ for $A$ 
 and 
 $\Gamma=\Gamma_0(2)/\{\pm \text{\rm Id}\}$, $z_0=(i-1)/2$ for $S$. In both cases, writing $y=x/m$, the function $h$ satisfies 
 \begin{equation}\label{hi2}
  h(i/2)=4\pi y 
  +
  O\big(  (y+y^{1/2})\Delta\big)
  \qquad\text{and}\qquad
  h(t)\ll H(t)\quad\text{for }t\in\R
 \end{equation}
 where $\Delta$ is an arbitrarily chosen number
  $0<\Delta<\min(1,y^{1/2})$, 
 \[
  H(t)= 
  y(1+yt^2)^{-3/4}
   \min\big(1, (\Delta|t|)^{-3/2}\big)
   \quad\text{if }0<y\le 1
 \]
 and 
 \[
  H(t)= 
  y^{1/2}(1+|t|)^{-3/2}
   \min\big(1, (\Delta|t|)^{-3/2}\big)
   \quad\text{if }y> 1\text{ and }|t|\ge 1.
 \]
 For $y>1$, $|t|<1$ the latter formula still holds multiplying the right hand side by $\log(2y)$.
\end{proposition}

We recall that in this context the Hecke operators are given by 
\begin{equation}\label{Heop}
 T_m f(z)
 =
 \frac{1}{\sqrt{m}}
 \sum_{ad=m}
 \sum_{b=0}^{d-1}
 f\Big(\frac{az+b}{d}\Big)
\end{equation}
and we have
\begin{equation}\label{HeMa}
 T_m u_j(z)
 =
 \lambda_j(m)u_j(z)
 \qquad\text{and}\qquad
 T_m E_{\mathfrak{a}}(z,1/2+it)
 =
 \eta_t(m)E_{\mathfrak{a}}(z,1/2+it).
\end{equation}
Here $\eta_t(m)=\sum_{ad=m}(a/d)^{it}$, hence $|\eta_t(m)|\le \tau(m)$. The Ramanujan-Petersson conjecture claims that $\lambda_j(m)$ is bounded in the same way and the state of the art loses a small power using profound techniques (see \cite{sarnakR} for the ideas leading to this breakthrough).

\begin{lemma}[{\cite[App.\,2]{kim}}]\label{kisal}
 We have $\theta\le 7/64$ in \eqref{lth}. In fact $|\lambda_j(p)|\le 2p^{7/64}$ for $p$ prime. 
\end{lemma}

An application of the Kuznetsov formula gives the following average result. For the proof of a stronger result see \cite[Cor.5.3]{ChRa} and use \cite[(8.43)]{iwaniec}.

\begin{lemma}\label{kuznetsov}
 For every $\epsilon>0$
 \[
  \sum_{T\le |t_j|<2T} 
  |\lambda_j(m)|^2
  \ll
  T^{2+\epsilon}+m^{1/2+\epsilon}.
 \]
\end{lemma}
\medskip

For Theorem~\ref{maini} we use the main result in \cite{IwSa} (see also \cite[\S10]{BlHo} and  \cite{sarnak}).
\begin{lemma}[{\cite[(A.15)]{IwSa}}]\label{Linf}
 For every $\epsilon>0$, 
 $|u_j(z_0)|\ll |t_j|^{5/12+\epsilon}$.
\end{lemma}
\medskip

We devote the rest of this section to explain the spectral origin of the main term in the asymptotic formula \eqref{S_asymp}.

The Fuchsian groups appearing in Proposition~\ref{spec} do not have exceptional eigenvalues \cite{huxley} then $t_j$ only takes real values (satisfying $|t_j|>5$ \cite{lmfdb}) and $t_0=i/2$ that corresponds to the zero eigenvalue $1/4+t_0^2$ of the Laplace-Beltrami operator coming from the constant eigenfunction $u_0$. 
As the fundamental domain of 
$\text{\rm PSL}_2(\Z)$ has area $\pi/3$ and $\Gamma_0(2)/\{\pm \text{\rm Id}\}$ is a subgroup of index 3, we have that $\lambda_0(m)|u_0(z)|^2$ is 
\[
 \frac{3\sigma(m)}{\pi\sqrt{m}}
 \quad\text{if }\Gamma=\text{\rm PSL}_2(\Z)
 \qquad\text{and}\qquad
 \frac{\sigma(m)}{\pi\sqrt{m}}
 \quad\text{if }\Gamma=\Gamma_0(2)/\{\pm \text{\rm Id}\}.
\]
If we tune $\Delta$ in such a way that $h(i/2)\sim 4\pi y$ and the contribution of the rest of the terms in the spectral expansion in Proposition~\ref{spec} is negligible then for $m$ fixed
\[
 A(4x,4m)\sim 
 \frac{24\sigma(m)}{m}x
 \text{ for }m\in\Z^+
 ,\quad
 S(x,m)\sim 
 \frac{8\sigma(m)}{m}x
 \text{ for }2\nmid m\in\Z^+.
\]
Lemma~\ref{SA} complements the latter formula producing for $m$ even the asymptotics $S(x,m)\sim 8\sigma(m/2)x/m$ if $k=1$ and 
\[
 S(x,m)\sim
 48
 \sum_{j=0}^{k-2}
 (-1)^j
 \sigma(m/2^{j+2})\frac xm
 +
 8(-1)^{k-1}
 \sigma(m/2^{k})\frac xm
 \quad\text{if }k>1.
\]
This equals $8\big(\sigma(2^k)-2\big)\sigma(m/2^{k}) x/m$ with standard manipulations  
using that~$\sigma$ is a multiplicative function. 
We can combine the three cases $2\nmid m$,
$4\nmid m$ with $m$ even and $4\mid m$ in the somewhat artificial single formula of~\eqref{S_asymp}.

\section{Proof of the main result}

After Proposition~\ref{spec} and knowing that the involved Eisenstein series behave essentially as the square of the Riemann zeta function, a fundamental problem to get uniform asymptotic formulas for $S(x,m)$ is to find good upper bounds for 
\begin{equation}\label{defE}
 \mathcal{S}(m,T)
 =
 \sum_{T\le |t_j|<2T} 
 |\lambda_j(m)| |u_j(z_0)|^2.
\end{equation}

Even if we employ exactly the same techniques as in \cite{chamizo_c}, the reduction given in \cite{kim} for the upper bound of $\theta$ changes substantially the way in which the optimization  can be made to obtain Theorem~\ref{main}, by this reason the unconditional results in \cite{chamizo_c} do not correspond to substitute $\theta\le 5/28$, the best upper bound available at that time, in Theorem~\ref{main}. 
There is also a new technique here not appearing in \cite{chamizo_c}. It consists in using the Hecke relation \cite[(8.39)]{iwaniec}
\begin{equation}\label{He_re}
 |\lambda_j(m)|^2
 =
 \sum_{d\mid m}
 \lambda_j\big(d^2\big)
\end{equation}
($\lambda_j\in\R$ because the Hecke operators are self-adjoint) to get via Cauchy's inequality 
\[
 \mathcal{S}^2(m,T)
 \le
 \sum_{d\mid m}
 \sum_{T\le |t_j|<2T} 
 \lambda_j\big(m^2/d^2\big) |u_j(z_0)|^2
 \cdot
 \sum_{T\le |t_j|<2T} 
 |u_j(z_0)|^2.
\]
The advantage of this expression is that for $d\ne m$ now $|u_j(z_0)|^2$ is multiplied by a changing sign coefficient that we can exploit using spectral theory to quantify the cancellation if $T$ and $m$ are not very large.
In \cite{chamizo_c} Cauchy's inequality was applied directly to \eqref{defE} and it required some knowledge about $\sum |u_j(z_0)|^4$. The alternative use of Cauchy's inequality described above was suggested by  Raphael Steiner (personal communication) and we fully credit him for this important remark that conveniently developed is responsible for the range $m^{3/2-\theta}\le x< m^{3/2+3\theta}$ in Theorem~\ref{main}.  

\begin{lemma}\label{newL}
 Let $\mathcal{S}(m,T)$ be as in \eqref{defE} and $\theta$ as in \eqref{lth}. Then
 \[
  \mathcal{S}(m,T)
  \ll
  m^\epsilon T^2
  \big(
  1+\min(m^\theta, m^{1/2}T^{-1})\big)
 \]
 for every $\epsilon>0$. 
\end{lemma}

\begin{proof}
 We vary a little the previous scheme, smoothing the sum and completing the spectrum. Namely, we start noting that $e^4\mathcal{S}(m,T)$ is less than
 \[
  \sum_j |\lambda_j(m)| e^{-t_j^2/T^2}|u_j(z_0)|^2
  +
  \frac{1}{4\pi}
  \sum_{\mathfrak{a}}
  \int_{\R}
  |\eta_{t}(m)|
  e^{-t^2/T^2}
  \big|E_{\mathfrak{a}}(z_0,\frac{1}{2}+it)\big|^2\; dt.
 \]
 Using Cauchy's inequality, \eqref{He_re} and recalling that $\eta_t(m)$ are also  Hecke eigenvalues obeying this relation, we have
 \begin{equation}\label{CauE}
  \mathcal{S}^2(m,T)
  \ll 
  \mathcal{S}_1(T)\sum_{d\mid m}
  \mathcal{S}_d(T)
 \end{equation}
 where 
 \[
    \mathcal{S}_d(T)
  =
  \sum_j \lambda_j(d^2) e^{-t_j^2/T^2}|u_j(z_0)|^2
  +
  \frac{1}{4\pi}
  \sum_{\mathfrak{a}}
  \int_{\R}
  \eta_{t}(d^2)
  e^{-t^2/T^2}
  \big|E_{\mathfrak{a}}(z_0,\frac{1}{2}+it)\big|^2\; dt.
 \]
 Note that $\lambda_j(1)=\eta_t(1)=1$. By \cite[Prop.7.2]{iwaniec} and \eqref{lth} we have the crude bound 
 \begin{equation}\label{crude}
  \mathcal{S}_d(T)
  \ll
  d^{2\theta+\epsilon}T^2.
 \end{equation}
 This is the best possible for $d=1$ but we expect some cancellation for larger values of $d$ at least in some ranges of $T$. 
 
 By the pretrace formula \cite[Th.7.4]{iwaniec} and \eqref{HeMa}
 \[
  \mathcal{S}_d(T)
  \ll
  T_{d^2}\Big|_{z=z_0}
  \sum_{\gamma\in\Gamma}
  k\big(u(\gamma z,z_0)\big)
  \qquad
  \text{where}\quad
  u(z,w)=\frac{|z-w|^2}{4\Im z \Im w}
 \]
 and $k(t)$ is 
 the inverse Selberg--Harish-Chandra transform of 
 $h(t)=e^{-t^2/T^2}$. As indicated, the Hecke operator $T_{d^2}$ acts on $z$. 
 
 Recall that the matrices corresponding to the maps $z\mapsto (az+b)/d$ in the definition of the Hecke operators \eqref{Heop} are representatives of $\Gamma\backslash\Gamma_{d^2}$ where~$\Gamma_{d^2}$ formally has the same definition as $\Gamma$ but imposing that the determinant of its elements is~$d^2$ instead of~$1$. Then 
 \[
  \mathcal{S}_d(T)
  \ll
  d^{-1}
  \sum_{\gamma\in\Gamma_{d^2}}
  k\big(u(\gamma z_0,z_0)\big).
 \]
 In \cite[\S5]{ChRa} it is shown a general estimate for the inverse Selberg--Harish-Chandra transform that gives in this case 
 \[
  k(t)\ll \phi(t)
  \qquad\text{with}\quad
  \phi\big(\sinh^2\frac{t}{2}\big)
  =
  T^2e^{-T^2t^2/4}
  \sqrt{\frac{t}{\sinh t}}.
 \]
 In particular $k(t)$ decays faster than any power and  as $\sinh^2({t}/{2})\sim t^2/4$ for~$t$ small, the main contribution comes from $u(\gamma z_0,z_0)\ll T^{-2+\epsilon}$. Let us define
 \[
  M(t)
  =
  \#\big\{
  \gamma\in\Gamma_{d^2}\;:\;
  u(\gamma z_0,z_0)<t
  \big\}.
 \]
 In \cite[(A.9), (A.10)]{IwSa} (see \cite[\S10]{BlHo} for a more explicit statement) there is a general bound for $M(t)$ based on a counting argument. We will show that we can take advantage of the special form of $z_0$, as in \cite[Lemma~13.1]{iwaniec}, to get 
 \begin{equation}\label{M_bound}
  M(t)
  \ll 
  d^\epsilon+
  \big(d^2t)^{1+\epsilon}.
 \end{equation}
 Hence
 \[
  \mathcal{S}_d(T)
  \ll
  d^{-1}
  T^2M(T^{-2})
  +d^{-1}
  \int_{T^{-2}}^\infty
  \phi(t)\,dM(t)
  \ll
  d^{-1+\epsilon}T^2 + d^{1+\epsilon}.
 \]
 Combining this with \eqref{crude} we deduce
 \[
  \mathcal{S}_d(T)
  \ll
  d^{\epsilon}
  T^2\big(d^{-1}+\min(d^{2\theta},dT^{-2})\big).
 \]
 And since \eqref{CauE}
 \[
  \mathcal{S}^2(m,T)
  \ll
  m^{\epsilon}T^4
  +
  m^{\epsilon}T^4
  \sum_{d\mid m}\min(d^{2\theta},dT^{-2})\big)
 \]
 which gives the result. 
 
 It remains to prove \eqref{M_bound}. The points $z_0=i$ and $z_0=(i-1)/2$ are in the same orbit under $\text{PSL}_2(\Z)$. If $z_0=(i-1)/2$ let $\gamma_0\in\text{PSL}_2(\Z)$ such that $z_0=\gamma_0 i$ then in the definition of $M$ we have $u(\gamma z_0,z_0)=u(\gamma \gamma_0 i,\gamma_0 i)=u(\gamma_0^{-1}\gamma \gamma_0 i,i)$
 with $\gamma_0^{-1}\gamma \gamma_0$ an integral matrix of determinant $d^2$. whence to prove \eqref{M_bound} we can restrict ourselves to the case $z_0=i$, $\Gamma= \text{PSL}_2(\Z)$. A calculation shows that if $\gamma$ has determinant $d^2$
 \[
  u(\gamma i,i)
  =
  \frac{(v-a)^2+(b+u)^2}{4d^2}
  \qquad\text{where}\quad
  \gamma= \begin{pmatrix}
           a & b \\ u & v
          \end{pmatrix}.
 \]
 Writing $A=v-a$, $B=b+u$, $C=a+v$ and $D=b-u$; $u(\gamma i,i)<t$ implies $A^2+B^2<4d^2t$. Noting $4d^2=C^2+D^2-A^2-B^2$ we have 
 \[
  M(t)\ll
  \sum_{A^2+B^2<4d^2t}
  r(4d^2+A^2+B^2)
  \ll
  \sum_{n<4d^2t}
  r(n)r(4d^2+n)
 \]
 and this shows \eqref{M_bound} because $r(k)=O(k^\epsilon)$.
\end{proof}

To ease references we state here the bound for $\mathcal{S}(m,T)$ obtained combining Lemma~\ref{kuznetsov} and a convexity bound coming from Lemma~\ref{Linf} and Bessel inequality. 

\begin{lemma}\label{KLinf}
 For every $\epsilon>0$
 \[
  \mathcal{S}(m,T)
  \ll 
  T^{17/12+\epsilon}\big(T+m^{1/4+\epsilon}\big).
 \]

\end{lemma}

\begin{proof}
 We have
 \[
  |\mathcal{S}(m,T)|^2
  \le
  \sum |\lambda_j(m)|^2
  \sum |u_j(z_0)|^4
  \le 
  \sum |\lambda_j(m)|^2
  \sum |u_j(z_0)|^2
  \sup |u_j(z_0)|^2
 \]
 where the sums and the supremum are over $T\le |t_j|<2T$. For the first sum use Lemma~\ref{kuznetsov}, for the second Lemma~\ref{newL} with $m=1$ (or directly Bessel inequality \cite[Prop.7.2]{iwaniec})
 and for the supremum use Lemma~\ref{Linf}.
\end{proof}
\medskip

We divide the proof of Theorem~\ref{main} in two parts studying separately the spectral contribution in the cases $m\le x$ and $m> x$.

We first state an auxiliary result bringing Proposition~\ref{spec} closer to the estimation of $E(x,m)$. 

\begin{lemma}\label{commonl}
 Let $\mathcal{E}(x,m)$ be the result of subtracting 
 $8\pi\sigma_{-1}(m)x/|\Gamma\backslash\mathbb{H}|$
 to the quantities $S(x,m)$ with $2\nmid m$ or $A(4x,4m)$. Then for $x\ge 1$
 \[
  \mathcal{E}(x,m)
  \ll
  m^\epsilon\Delta
  \big(x+(mx)^{1/2}\big)
  +x^{1/2}\log(2x)
  +
  \sum_{j=2}^\infty
  \mathcal{E}_{2^j}(x,m)
 \]
 where $\mathcal{E}_{2^j}(x,m)=\sqrt{m}H(T)
 \big(|\mathcal{S}(m,T)|+T^2\big)$.
\end{lemma}

Our bounds for $\mathcal{S}(m,T)$ are greater than $T^2$, even under Ramanujan-Petersson conjecture, hence the term $T^2$ in $\mathcal{E}_T$ is irrelevant in practice. 

\begin{proof}
 The normalized constant eigenfunction $u_0(z)=|\Gamma\backslash\mathbb{H}|^{-1/2}$ has Hecke eigenvalue $\sigma(m)/\sqrt{m}=O(m^{1/2+\epsilon})$. Then its contribution to the spectral expansion in Proposition~\ref{spec} is, according to \eqref{hi2},
 \[
  2\sqrt{m}
  \frac{\sigma(m)}{\sqrt{m}}|u_0(z_0)|^2
  \big(4\pi y +O(\Delta (y+y^{1/2}))\big)
 \]
 where $y=x/m$. Subtracting $8\pi\sigma(m)y/|\Gamma\backslash\mathbb{H}|$ we get
 $O\big(m^\epsilon \Delta (x+(mx)^{1/2})\big)$.
 As we mentioned before, $|t_j|>5$ for the involved groups~\cite{lmfdb}. Then the sum of $\mathcal{E}_{2^j}$ for $j\ge 2$ bounds the contribution of the discrete spectrum in Proposition~\ref{spec} subdividing into dyadic intervals. For the Eisenstein series it is known \cite[Th.6.2]{muller} that 
 $\int_{T}^{2T}\big|E_{\mathfrak{a}}(z,\frac{1}{2}+it)\big|^2\; dt\ll T(\log T)^2$
 then the term~$T^2$ in $\mathcal{E}_{T}$ absorbs their contribution for $|t|>4$. Finally
 \[
  \int_{-4}^4
  \eta_{t}(m)
  h(t)
  \big|E_{\mathfrak{a}}(z_0,\frac{1}{2}+it)\big|^2\; dt
  \ll
  \sqrt{m}
  \int_{-4}^4
  H(t)\; dt
 \]
 and this is $O\big(x^{1/2}\log(2x)\big)$ using the crude bound $H(t)\ll (x/m)^{1/2}\log(2x)$ if $x>m$ and $H(t)\ll x/m$ if $x\le m$.
\end{proof}

\begin{proposition}\label{mlex}
 Let $\mathcal{E}(x,m)$ be as in Lemma~\ref{commonl}. Then for $1\le m\le x$ and any $\epsilon>0$ we have
 \[
  m^{-\epsilon}\mathcal{E}(x,m)
  \ll
  \begin{cases}
   x^{2/3} & \text{if } m^{3/2+3\theta}\le x,
   \\
   m^{(1+2\theta)/4}x^{1/2} & \text{if } m^{3/2-\theta}\le x<m^{3/2+3\theta},
   \\
   m^{2\theta/3}x^{2/3} & \text{if } m\le x<m^{3/2-\theta}.
  \end{cases}
 \]
\end{proposition}

\begin{proof}
 We have
 $\sqrt{m}H(T)\ll x^{1/2}T^{-3/2}\min\big(1,(T\Delta)^{-3/2}\big)$
 in Lemma~\ref{commonl} and substituting the bound of Lemma~\ref{newL}, we get
 \[
  \mathcal{E}_T(x,m)
  \ll
  m^\epsilon 
  x^{1/2}T^{1/2}\min\big(1,(T\Delta)^{-3/2}\big)
  \big(
  1+\min(m^\theta, m^{1/2}T^{-1})\big).
 \]
 Defining
 \[
  F(\Delta, T)
  =
  x\Delta +
  x^{1/2}T^{1/2}\min\big(1,(T\Delta)^{-3/2}\big)
  \big(
  1+\min(m^\theta, m^{1/2}T^{-1})\big)
 \]
 it is enough to prove 
 \begin{equation}\label{infsupF1}
  \inf_{0<\Delta<1}
  \sup_{T\ge 4} 
  F(\Delta, T)
  \ll 
  \begin{cases}
   x^{2/3}+m^{1/4+\theta/2}x^{1/2}&\text{if } x\ge m^{3/2-\theta},
   \\
   m^{2\theta/3}x^{2/3}&\text{if } x< m^{3/2-\theta}
  \end{cases}
 \end{equation}
 because $x^{2/3}\ge m^{1/4+\theta/2}x^{1/2}$ if and only if $x\ge m^{3/2+3\theta}$. Choosing $T$ and $\Delta$ as in the rest of the proof one could show that \eqref{infsupF1} is in fact sharp, the ``$\ll$'' sign could be replaced by ``$\asymp$''.
 
 The values $\Delta^{-1}$, $m^{1/2-\theta}$ and $m^{1/2}$ subdivide $[4,\infty)$ into at most~$4$ intervals. On each of them $F(\Delta,T)-x\Delta$ behaves as a power of $T$ then   $G(\Delta)=\sup_{T\ge 4} F(\Delta, T)$ satisfies
 \[
  G(\Delta)
  \asymp 
  F(\Delta,\Delta^{-1})
  +F(\Delta,m^{1/2-\theta})
  +F(\Delta,m^{1/2}).
 \]
 Clearly $\min\big(1,(T\Delta)^{-3/2}\big)$ is not increasing in $T$ and $T^{1/2}\min(m^\theta, m^{1/2}T^{-1})$
 is greater for $T=m^{1/2-\theta}$ than for $T=m^{1/2}$. It assures  $F(\Delta,m^{1/2})\ll F(\Delta,m^{1/2-\theta})$ and we have 
 \begin{multline*}
  G(\Delta)
  \ll
  x\Delta
  +
  x^{1/2}\Delta^{-1/2}\big(
  1+\min(m^{1/2}\Delta, m^\theta)\big)
  \\
  +m^{(1+2\theta)/4}x^{1/2}  
  \min\big(1,(m^{1/2-\theta}\Delta)^{-3/2}\big).
 \end{multline*}
 If $x\ge  m^{3/2-\theta}$ we choose the first arguments in the minima to get 
 \[
  G(\Delta)
  \ll
  x\Delta
  +
  x^{1/2}\Delta^{-1/2}
  +
  m^{1/2}x^{1/2}\Delta^{1/2}
  +
  m^{(1+2\theta)/4}x^{1/2}.
 \]
 Hence $\inf G(\Delta)\le G(x^{-1/3})\ll x^{2/3}+m^{1/2}x^{1/3}+m^{(1+2\theta)/4}x^{1/2}$ and the central term is negligible because $x^{2/3}>m^{1/2}x^{1/3}$ if $x<m^{3/2}$ and we have  $m^{1/2}x^{1/3}<m^{(1+2\theta)/4}x^{1/2}$ otherwise.
 
 If $x< m^{3/2-\theta}$ we choose the second arguments in the minima to deduce 
 \[
  G(\Delta)
  \ll
  x\Delta
  +
  m^{\theta}x^{1/2}\Delta^{-1/2}
  +
  m^{-1/2+2\theta}x^{1/2}\Delta^{-3/2}.
 \]
 Hence $\inf G(\Delta)\le G(m^{2\theta/3}x^{-1/3})\ll m^{2\theta/3}x^{2/3}+m^{-1/2+\theta}x$ and in our range the last term is negligible. 
\end{proof}

\begin{proposition}\label{mgex}
 Let $\mathcal{E}(x,m)$ be as in Lemma~\ref{commonl}. Then given $\epsilon>0$ for $x\le m< x^{2/(1+2\theta)}$  we have $\mathcal{E}(x,m)
  \ll m^{(1+2\theta)/3+\epsilon}x^{1/3}$.
\end{proposition}

\begin{proof}
 In this case we have $H(T)\ll m^{-1}xT^{2}(1+m^{-1}xT^2)^{-3/4}\min\big(1,(T\Delta)^{-3/2}\big)$
 in Lemma~\ref{commonl} and the bound of Lemma~\ref{newL} gives
 \begin{multline*}
  \mathcal{E}_T(x,m)
  \ll
  m^{-1/2+\epsilon} 
  x T^{2}\big(1+xT^2m^{-1}\big)^{-3/4}
  \min\big(1,(T\Delta)^{-3/2}\big)\cdot
  \\
  \big(1+\min(m^\theta, m^{1/2}T^{-1})\big).
 \end{multline*}
 Using $\big(1+xT^2m^{-1}\big)^{-3/4}\ll m^{3/4}x^{-3/4}T^{-3/2}$, the function to be optimized is 
 \begin{multline*}
  F(\Delta,T)
  =
  m^{1/2}x^{1/2}\Delta 
  \\
  +
  m^{1/4}x^{1/4}T^{1/2}\min\big(1,(T\Delta)^{-3/2}\big)
  \big(
  1+\min(m^\theta, m^{1/2}T^{-1})\big).
 \end{multline*}
 Namely, we have to prove 
 \begin{equation}\label{infsupF2}
  \inf_{0<\Delta<(x/m)^{1/2}}
  \sup_{T\ge 4} 
  F(\Delta, T)
  \ll 
  m^{(1+2\theta)/3}x^{1/3}.
 \end{equation}
 The bound $\Delta<(x/m)^{1/2}$ is required by Proposition~\ref{spec} and it is less important for the optimization. 
 Note that $F(\Delta,T)-m^{1/2}x^{1/2}\Delta$ is like in the proof of Proposition~\ref{mlex} except for a coefficient not depending on $\Delta$ and $T$. Hence the same argument applies to show that the supremum on $T$ in \eqref{infsupF2}, say $G(\Delta)$, satisfies 
 \begin{multline*}
  G(\Delta)
  \ll 
  m^{1/2}x^{1/2}\Delta 
  +
  m^{1/4}x^{1/4}\Delta^{-1/2}
  \big(
  1+\min(m^\theta, m^{1/2}T^{-1})\big)
  \\
  +
  m^{(1+\theta)/2}x^{1/4}\min\big(1,(m^{1/2-\theta}\Delta)^{-3/2}\big).
 \end{multline*}
 
 In the last range of the case $m\le x$ in the proof of Proposition~\ref{mlex}, the term $m^\theta$ gave the minimum and was of greater order than $m^{1/2}\Delta^{1/4}$ then it is natural to consider by continuity that this is still the situation. With this idea in mind, let us take
 $\Delta^{-1}=x^{1/6}m^{1/6-2\theta/3}$ to balance the corresponding terms. Note that in our range it satisfies the assumption $\Delta^{-1}>m^{1/2}x^{-1/2}$. The result is 
 that  \eqref{infsupF2} holds with 
 $m^{(1+2\theta)/3}x^{1/3}$ accompanied with the extra terms
 \[
  m^{(1-\theta)/3}x^{1/3}
  \min(m^\theta, m^{(1+2\theta)/3}x^{-1/6})
  +
  m^{(1+\theta)/2}x^{1/4}
  \min(1, m^{(\theta-1)/2}x^{1/4}).
 \]
 In our range the minima are $m^\theta$ and $m^{(\theta-1)/2}x^{1/4}$, respectively. Hence  it only remains to check that $m^{(1+2\theta)/3}x^{1/3}\ge m^{\theta}x^{1/2}$ which follows from $m\ge x$.
\end{proof}

Finally, we prove the estimate required for Theorem~\ref{maini}.

\begin{proposition}\label{impr}
 Let $\mathcal{E}(x,m)$ be as in Lemma~\ref{commonl} and $x$ and $m$ in the range indicated in Theorem~\ref{maini}. Then for every $\epsilon>0$
 \[
  \mathcal{E}(x,m)
  \ll
  x^{17/23+\epsilon}
  +
  (mx)^{17/46+\epsilon}
  +
  m^{(13+4\theta)/28}x^{1/4+\epsilon}.
 \]
\end{proposition}

\begin{proof}
 Let us consider first the case $m<x$ that only occurs for $\theta>5/46$ because $x<m^{46\theta/5}$.
 We have 
 $\sqrt{m} H(T)
 \ll x^{1/2}T^{-3/2}\min\big(1, (\Delta T)^{-3/2}\big)$.
 Let us take $\Delta=x^{-6/23}$ in Lemma~\ref{commonl} and use Lemma~\ref{KLinf}. If we prove
 \[
  \sqrt{m}H(T) T^{17/12}(T+m^{1/4})
  \ll x^{17/23}
 \]
 then, since the left hand side has a potential behavior in~$T$ and eventually decays to zero, we deduce $\mathcal{E}(x,m)\ll x^{17/23+\epsilon}$ which fulfills our needs. Note that $\Delta^{-1}>m^{1/4}$ then the supremum only may be attached at $T\asymp \Delta^{-1}$ or $T\asymp m^{1/4}$. Adding both contributions we get 
 $x^{17/23}+m^{11/48}x^{1/2}$ and this is $O(x^{17/23})$ because $m<x$. 
 \smallskip
 
 Now we deal with the case $x\le m$. 
 Using $1+yt^2> yt^2$ in Proposition~\ref{spec} we have 
 $\sqrt{m} H(T)
 \ll (mx)^{1/4}T^{-3/2}\min\big(1, (\Delta T)^{-3/2}\big)$.
 Let us take $\Delta^{-1}=(mx)^{3/23}$ (it is easy to check that in our range $\Delta<y^{1/2}$ is fulfilled in Proposition~\ref{spec}). 
 Combining Lemma~\ref{newL} and Lemma~\ref{KLinf} to bound $\mathcal{E}(m,T)$ in Lemma~\ref{commonl}, it is enough to prove, as before,
 \begin{equation}\label{aux_imp}
  \sqrt{m}
  H(T)
  \min\big(m^\theta T^2, T^{17/12}(T+m^{1/4})\big)
  \ll
  (mx)^{17/46+\epsilon}
  +
  m^{(13+4\theta)/28}x^{1/4+\epsilon}
 \end{equation}
 for every $T\ge 4$. 
 We can, of course, replace $T+m^{1/4}$ by $\max(T,m^{1/4})$ and the values
 \[
  T_0=m^{1/4},\quad
  T_1=m^{3(1-4\theta)/7},\quad
  T_2=\Delta^{-1}\quad\text{and}\quad
  T_3=m^{12\theta/5},
 \]
 coming respectively from the equations $T_0=m^{1/4}$, $m^\theta T^2=T^{17/12}m^{1/4}$, $1=T\Delta$ and $m^\theta T^2=T^{29/12}$, mark the points at which there is a possible change in the dominant value of the maximum or the minima. In the rest of the cases, the left hand side of \eqref{aux_imp} is comparable to a function of the form $c(x,m)T^\alpha$, $\alpha\ne 0$. this is strictly monotonic in the intervals determined by the $T_j$ and then it is enough to check  \eqref{aux_imp} for $T=T_0,\dots, T_3$. 
 
 In our ranges a calculation shows that $T_3\ge (mx)^{3/23}=\Delta^{-1}$ then we can substitute
 $\min\big(1, (\Delta T_3)^{-3/2}\big)$ by $(\Delta T_3)^{-3/2}$. Hence
 \[
  \sqrt{m}H(T_3)m^\theta T_3^2
  \ll
  m^{41/92-7\theta/5}
  x^{41/92}
  =
  (mx)^{17/46}
  \big(m^{1-92\theta/5} x)^{7/92}
 \]
 and this is less than $(mx)^{17/46}$ because $x<m^{(92\theta-5)/5}$. 
 
 For $T_0$, $T_1$ and $T_2$ we use 
 $\min\big(1, (\Delta T_3)^{-3/2}\big)\le 1$. In this way, for $T_1$
 \[
  \sqrt{m}H(T_1)m^\theta T_1^2
  \ll
  m^{(13+4\theta)/28}x^{1/4}
 \]
 which is part of the right hand side of \eqref{aux_imp}.
 
 Finally, for $T_0$ and $T_2$ we choose $T^{17/12}(T+m^{1/4})$ in the minimum. For~$T_0$
 \[
  \sqrt{m}H(T_0)T_0^{17/12}(T_0+m^{1/4})
  \ll
  m^{23/48}x^{1/4}
 \]
 and $23/48\le (13+4\theta)/28$ because $\theta\ge 5/48$. For $T_2$, a calculation shows 
 \begin{multline*}
  \sqrt{m}H(T_2)T_2^{17/12}(T_2+m^{1/4})
  \ll
  (mx)^{1/4}
  \Delta^{1/12}\big(\Delta^{-1}+m^{1/4}\big)
  \\
  =
  (mx)^{17/46}+m^{45/92}x^{11/46}.
 \end{multline*}
 The last term can be written as $(m^{45}/x)^{1/92}x^{1/4}$. Since $x>m^{11(1-4\theta)/7}$, 
 $m^{45}/x<m^{4(76+11\theta)/7}$ 
 and 
 $4(76+11\theta)/7\le 92(13+4\theta)/28$ because $\theta\ge 5/48$. 
\end{proof}

\begin{proof}[Proof of Theorem~\ref{main} and Theorem~\ref{maini}]
 If $m$ is odd then 
 $S(x,m)=8\sigma(m)x/m+\mathcal{E}(x,m)$
 and the result follows directly from Proposition~\ref{mlex}, Proposition~\ref{mgex} and Proposition~\ref{impr}.
 If $m$ is even, we use Lemma~\ref{SA}. By the comments at the end of section~\ref{prelim} we know that the main terms coming from $\widetilde{A}$ give rise to the main term in~\eqref{S_asymp} and then 
 \[
  E(x,m)=
  2
  \Big(
  \mathcal{E}\big(\frac{x}{2^2},\frac{m}{2^2}\big)
  -
  \dots
  +(-1)^{k-2}
  \mathcal{E}\big(\frac{x}{2^k},\frac{m}{2^k}\big)
  \Big)
  +(-1)^{k-1}
  \mathcal{E}\big(\frac{x}{2^k},\frac{m}{2^k}\big)
 \]
 where the last $\mathcal{E}$ corresponds to the group
 $\Gamma=\Gamma_0(2)/\{\pm \text{\rm Id}\}$
 and the rest to~$\Gamma=\text{\rm PSL}_2(\Z)$.
 
 In Proposition~\ref{mlex}, Proposition~\ref{mgex} and Proposition~\ref{impr} all the exponents in the bounds are positive then we have a geometric gain replacing $(x,m)$ by $(x/2,m/2)$ when we stay in the same rang. By the continuity of the bound (recall the figures showing the graphs), we get comparable   results each time we change the range. As there are only a finite number of ranges, we conclude that the bounds for $\mathcal{E}(x,m)$ are also valid for $E(x,m)$.
 Note that the range required in Proposition~\ref{mgex} can be assumed because otherwise $E(x,m)$ supersedes the main term and $S(x,m)=O(x^{1+\epsilon})$ follows trivially from $r(n)=O(n^\epsilon)$. 
\end{proof}

\section{Comparison with other results and further comments}

In \cite{chamizo_c} it is studied the asymptotic behavior of $S(x,m)$ under the conjecture
\begin{equation}\label{conj4}
 \sum_{T\le |t_j|<2T}
 |u_j(z_0)|^4
 =
 O\big(T^{2+\epsilon}\big)
 \qquad\text{for every }\epsilon>0.
\end{equation}
As the matter of fact something similar is known in the so-called $q$-aspect \cite{blomer}.

Unfortunately the main result in \cite{chamizo_c} is not written explicitly in terms of $\theta$ but it was substituted the upper bound 
available at that time, $\theta\le 5/28$. If we replace along its proof $5/28$ by $\theta$, we would get

\begin{theorem}[cf. {\cite[Th.3.1]{chamizo_c}}]\label{oldth}
 With the notation of Theorem~\ref{main} and under \eqref{conj4}, we have 
 \[
  x^{-\epsilon}E(x,m)
  \ll
  x^{2/3}
  +
  x^{1/2}m^{(1+4\theta)/8}
  +
  x^{1/3}m^{1/3}
  +
  \min\big(
  x^{1/2}m^{1/4}
  ,
  x^{1/4}m^{(3+4\theta)/8}
  \big)
 \]
 for every $\epsilon>0$ and $1\le m\le x^2$.
\end{theorem}

It turns out that this bound coincides (disregarding~$\epsilon$) with the best known bound for the binary additive problem included in \cite{meurman}. 
This famous problem, connected to the 4th moment of the Riemann zeta function \cite{HeBr}, deals with the sums
\[
 D(x,m)= \sum_{n\le x}\tau(n)\tau(n+m).
\]
The coincidence of the bounds is quite surprising because the methods appearing in \cite{meurman} are not entirely spectral, namely it is employed a combinatorial decomposition of the divisor function introduced by Heath-Brown. 
Recently in \cite{BaFr} it has been proved that a purely spectral approach is possible with a modification of \cite{motohashi}. 
Although the analogy between $r(n)$ and $\tau(n)$ is quite natural because the former is like a divisor function in arithmetic progressions, it is puzzling that the proof of Theorem~\ref{oldth} is much simpler than that of the main results of \cite{meurman} or \cite{motohashi} and still leads to the same result under \eqref{conj4}. On the other hand \eqref{conj4} has not an analogue in these papers which are unconditional. It vaguely suggests that this conjecture admits a proof via the Kuznetsov formula, which is the core argument in \cite{motohashi} and \cite{meurman}. Another possible approach to \eqref{conj4} is to use a lift (a Waldspurger like formula) to turn the sum of \eqref{conj4} at the considered special points into a sum of special values of automorphic $L$-functions (this is related to \cite{PeYo} and \cite{young} mentioned below, see also \cite{KaSa} for an application of this idea).

A very vague argument suggesting that \eqref{conj4} could be affordable relies on the general bound $\|\phi\|_\infty\le \lambda^{1/4}$ \cite{SeSo} for a normalized eigenfunction $\phi$ with eigenvalue $\lambda$. In our case this is $|u(z)|\ll T^{1/2}$ and it marks a kind of barrier for the difficulty (the ultimate conjecture is $|u(z)|\ll T^{\epsilon}$, see \cite{sarnak_aqc}). It can be obtained from the average result $\sum_{T\le |t_j|<T+1}|u(z)|^2=O(T)$ dropping all the terms except one and from \eqref{conj4} except for~$\epsilon$, while higher power analogues would break the barrier of the difficulty. Although $\sum_{T\le |t_j|<T+1}|u(z)|^2=O(T)$ and \eqref{conj4} have the same $L^\infty$ implications, the former follows from pretrace formula while our attempts to get a spectral proof of \eqref{conj4} have been unsuccessful so far. 
\medskip

An important point to emphasize is that under \eqref{conj4} we have the uniformity in \eqref{S_asymp} for
$m=O\big(x^{2-\epsilon}\big)$
even using results much weaker than Lemma~\ref{kisal} because the relevant term in Theorem~\ref{oldth} is~$x^{1/2}m^{1/4+\epsilon}$.

A natural question is to what extent we can relax \eqref{conj4} to improve the unconditional range in Corollary~\ref{c64_39}. 
This question admits a neat answer.

\begin{theorem}
 Let $\beta\ge 2$ such that
 \begin{equation}\label{conj4beta}
  \sum_{T\le |t_j|<2T}
  |u_j(z_0)|^4
  =
  O\big(T^{\beta+\epsilon}\big)
  \qquad\text{for every }\epsilon>0.
 \end{equation}
 Then the asymptotic formula \eqref{S_asymp} holds for $m=m(x)$ satisfying $m=O\big(x^{\eta}\big)$ for any $\eta<\beta/(\beta-1)$. 
\end{theorem}

To our knowledge, the best known exponent in \eqref{conj4beta} is $\beta=17/6$ deduced by convexity from Lemma~\ref{Linf} as in Lemma~\ref{KLinf}. 
Actually the main result in the recent preprint \cite{PeYo} might improve this exponent via \cite[(2.1), (2.2)]{young} but it is unclear if this is only valid for $z_0=i$. 
See also \cite{steiner} for an advance in the case of $S^3$ which is a completely different setting with a somehow similar flavor. 

\begin{proof}
 We assume $m>x^{64/39-\epsilon}$ by Corollary~\ref{c64_39} and consequently~$\beta<3$.  
 
 By Lemma~\ref{kuznetsov}, \eqref{conj4beta} and Cauchy's inequality,
 \[
  \mathcal{S}(m,T)
  \ll 
  T^{\beta/2+\epsilon}
  \big(T+m^{1/4+\epsilon}\big).
 \]
 Taking $\Delta= x^{1/2}m^{-1/2-\epsilon}$
 in Lemma~\ref{commonl}, the uniformity is assured if
 \[
  (mx)^{1/4}
  T^{-3/2}
  \min\big(1,(\Delta T)^{-3/2}\big)
  T^{\beta/2}\big(T+m^{1/4}\big)
  \ll x^{1-\delta}
 \]
 for some $\delta>0$ and any $T>4$. The maximum of the left hand side is attached at $T\asymp m^{1/4}$ or $T\asymp \Delta^{-1}$, then we have uniformity whenever 
 \[
  (mx)^{1/4}
  m^{-3/8}
  m^{\beta/8+1/4}
  +
  (mx)^{1/4}
  \Delta^{(3-\beta)/2}
  \big(\Delta^{-1}+m^{1/4}\big)
  \ll x^{1-\delta}.
 \]
 The first term imposes the condition $m=O\big(x^{\eta}\big)$ with  $\eta<6/(\beta+1)$ that is satisfied when  $\eta<\beta/(\beta-1)$ because~$2\le \beta<3$. In our range $\Delta^{-1}+m^{1/4}\ll m^{1/4}$ then it only remains to consider 
 $m^{1/2}x^{1/4}\Delta^{(3-\beta)/2}\ll x^{1-\delta}$
 and a calculation shows that this is equivalent to $m=O\big(x^{\eta}\big)$ with  $\eta<\beta/(\beta-1)$.
\end{proof}
\medskip

\paragraph{\textbf{Acknowledgements.}}
I am deeply indebted to R. Steiner for sharing his thoughts about~\eqref{conj4}, providing some references and suggesting the application of Cauchy's inequality in the form employed in the proof of Lemma~\ref{newL}.
I want to show my gratitude especially to E.~Valenti for the support  and tireless patience in these difficult times.

\providecommand{\bysame}{\leavevmode\hbox to3em{\hrulefill}\thinspace}
\providecommand{\MR}{\relax\ifhmode\unskip\space\fi MR }
\providecommand{\MRhref}[2]{%
  \href{http://www.ams.org/mathscinet-getitem?mr=#1}{#2}
}
\providecommand{\href}[2]{#2}

\end{document}